\documentclass[11point]{amsart}
\usepackage{amsmath}
\usepackage{amscd}

\usepackage{amssymb,amsmath,amscd,epsfig,amsthm,enumerate,import, comment,color}

\usepackage{geometry}                
\newtheorem{theorem}{Theorem}[section]

\newtheorem{lemma}[theorem]{Lemma}

\theoremstyle{definition}

\theoremstyle{remark}
\newtheorem{remark}[theorem]{Remark}

\numberwithin{equation}{section}
\newcommand{\be}{\begin{equation}}
\newcommand{\ee}{\end{equation}}

\begin{document}

\title{Spectral Properties of Semi-classical Toeplitz Operators}
\author{V. Guillemin}
\address{Mathematics Department\\
Massachusetts Institute of Technology\\Cambridge, MA 02138}
\email{vwg@math.mit.edu}
\author{A. Uribe}
\address{Mathematics Department\\
University of Michigan\\Ann Arbor, Michigan 48109}
\email{uribe@umich.edu}
\author{Z. Wang}
\address{School of Mathematical Sciences\\
University of Science and Technology of China \\ Hefei, Anhui 230026 \\ P. R. China}
\email{wangzuoq@ustc.edu.cn}

\thanks{Z.W. is partially supported by the National Science Foundation of China, Grant No.  11571131 and  Grant No. 11526212. 
}

\date{\today}
\begin{abstract}
The main results of this paper are an asymptotic expansion in powers of $\hbar$ for the spectral measure $\mu_\hbar$ of a semi-classical Toeplitz operator, $Q_\hbar$, and an equivariant version of this result when $Q_\hbar$ admits an $n$-torus as a symmetry group. In addition we discuss some inverse spectral consequences of these results. 
\end{abstract}

\maketitle
\tableofcontents

\section{Introduction}
 
 Before we attempt to explain the ``semi-classical" in the title above, we will first review a few basic facts about the classical theory of Toeplitz operators (the topic of the Princeton series monograph \cite{BG}). To define these operators, let $X$ be a compact manifold and $\Sigma$ a closed symplectic cone in $T^*X-0$. To this cone one can attach a Hilbert subspace, $H^2(X)$, of $L^2(X)$ with the property that the orthogonal projection 
\begin{equation}\label{orthoproj}
\Pi: L^2(X) \to H^2(X)
\end{equation}
is a Fourier integral operator of Hermite type (in the sense of \cite{BG}, \S 3) with microsupport on the diagonal $\Delta_\Sigma$ in $\Sigma \times \Sigma$. The prototypical example of such an operator is the classical \emph{Szeg\"o projector}. Namely in this example $X$ is the boundary of a strictly pseudoconvex domain, $H^2(X)$ is the $L^2$-closure of the kernel of the $\bar \partial_b$ operator and $\Sigma$ is the characteristic variety of this operator. (For a slightly simplified version of this example see \S 2 below.) 
 
The Toeplitz operators we'll be concerned with in this paper are operators of the form 
\begin{equation}\label{ToepOpe}
Q=\Pi P \Pi,
\end{equation} 
 where $P$ is a classical pseudodifferential operator of order $d$ on $X$ whose symbol will be denoted by $p$. One of the main result of \cite{BG}, a result which we will make heavy use of below, is that in defining these operators one can assume without loss of generality that $P$ commutes with $\Pi$. (Thus in particular these operators form an algebra.)
 
 To return to the ``semi-classical" in the title of this article suppose that there is a free action of the circle group $S^1$ on $X$ and that the induced action on $L^2(X)$ preserves $H^2(X)$ and commutes with $P$. Then if we denote by $L^2(X)_k$ the subspace of functions $f \in L^2(X)$ satisfying 
 \begin{equation}
 \label{L2Xk}
 f(e^{i\theta}x)=e^{ik\theta}f(x),
 \end{equation}
$P$ preserves this space and also preserves the intersection, 
\[ H^2(X)_k = L^2(X)_k \cap H^2(X).\]
Moreover, the operator 
\begin{equation}\label{Phk}
P_\hbar = \hbar^l P|_{L^2(X)_k}, \qquad \hbar=\frac 1k
\end{equation}
can be viewed as a semi-classical pseudodifferential operator of degree $d-l$, and the operator 
\begin{equation}\label{Qhk}
Q_\hbar=P_\hbar|_{H^2(X)_k}
\end{equation}
as a semi-classical Toeplitz operator of degree $d$. 

In the result we're about to describe the symbols of these operators play a major role, and to describe these symbols let 
\[
\phi: T^*X \to \mathbb R
\]
be the moment map associated with the lifted action of $S^1$ on $T^*X$ and let
\begin{equation}\label{TstarXred}
(T^*X)_{red}=\phi^{-1}(1)/S^1
\end{equation}
be the symplectic reduction of $T^*X$ at $\phi=1$. Then since $p|_{\phi^{-1}(1)}$ is $S^1$ invariant it is the pull-back to $\phi^{-1}(1)$ of a function 
\begin{equation}
p_{red}: (T^*X)_{red} \to \mathbb R
\end{equation}
and this reduced symbol is the semi-classical symbol of $P_\hbar$. Moreover the quotient 
\begin{equation}\label{sigmared}
\Sigma_{red} = \phi^{-1}(1) \cap \Sigma /S^1
\end{equation}
sits inside $(T^*X)_{red}$ and the restriction of $p_{red}$ to $\Sigma_{red}$ is the semi-classical symbol of $Q_\hbar$. 

A few final assumptions: We will henceforth assume $P$ is a self-adjoint elliptic pseudodifferential operator of degree $d$ with positive leading symbol $p$. Therefore, since $X$ is compact, these assumptions imply that the spectrum of $P$ is discrete, bounded from below and has no finite points of accumulation and hence that the same is true of $P_\hbar$ and $Q_\hbar$. 

Also one additional assumption: In order for these operators to be of order zero we will assume that the $l$ in (\ref{Phk}) is equal to $d$.

We will now describe the main results of this paper. Let $\mu_\hbar$ be the spectral measure
\begin{equation}\label{specmeas}
\mu_\hbar(f):= \mathrm{trace} f(Q_\hbar), \qquad \forall f \in C_0^\infty(\mathbb R).
\end{equation}
We will prove
\begin{theorem}\label{thm1}
The spectral measure $\mu_\hbar$ admits an asymptotic expansion 
\begin{equation}\label{muhexpansion}
\mu_\hbar(f) \sim (2\pi \hbar)^{-m} \sum_{i=0}^\infty c_i(f) \hbar^i
\end{equation}
with $m= \frac 12 {\dim \Sigma_{red}} $ and 
\begin{equation}\label{c0f}
c_0(f) = \int_{\Sigma_{red}} (p_{red})^*f d\sigma,
\end{equation}
$\sigma$ being the symplectic volume form on $\Sigma_{red}$. 
\end{theorem}

We will also use this result to prove a slightly more general result. Suppose one has an action on $X$ of an $n$-torus $G$, and that this action commutes with $\Pi$. Then if the operator $P$ in (\ref{ToepOpe}) is $G$-invariant and $\alpha$ is an element of the weight lattice of $G$, the action of $G$ on $H^2(X)$ preserves the space of functions in $H^2(X)$ which transform under this action by the weight $k\alpha$. Denoting the space of these functions by $H^2_{k\alpha}$ and by $P_{k\alpha}$ the restriction of $P$ to this space. We will prove an analogue of Theorem 1 for the spectral measure of $P_{k\alpha}$: Since the action of $G$ preserves $\Pi$, it preserves the microsupport $\Sigma$ of $\Pi$ and hence induces a Hamiltonian action on $\Sigma$ with moment map 
\[
\phi: \Sigma \to \mathfrak g^*.
\]

Assume now that $\alpha$ is a regular value of this moment map and that $G$ acts freely on the level set $\phi^{-1}(\alpha)$ (the assumptions required for the symplectic reduction of $\Sigma$ at $\alpha$
\begin{equation}
\label{sigmaalpha}
\Sigma_\alpha = \phi^{-1}(\alpha)/G
\end{equation}
to be well-defined.)  We will prove
\begin{theorem}\label{thm2}
The spectral measure of $\Pi P_{k\alpha}\Pi$, $k = \frac 1\hbar$, admits an asymptotic expansion 
\begin{equation}
\label{PiPkalphaPiExp}
(2\pi \hbar)^{-m} \sum_{i=1}^\infty \mu_{\alpha, i} \hbar^i
\end{equation}
where $m=\frac 12{\dim \Sigma_\alpha}$ and 
\begin{equation}\label{mualpha0}
\mu_{\alpha, 0}(f)=\int_{\Sigma_\alpha} p_\alpha^*f d\sigma,
\end{equation}
$\sigma$ being the symplectic volume form on $\Sigma_\alpha$ and $p_\alpha$ the ``reduced" semi-classical symbol of $P$, i.e. the map $p_\alpha: \Sigma_\alpha \to \mathbb R$ defined by 
\begin{equation}
\gamma^*p_\alpha = p|_{\phi^{-1}(\alpha)},
\end{equation}
where $\gamma$ is the projection of $\phi^{-1}(\alpha)$ onto $\Sigma_\alpha$.
\end{theorem}

\begin{remark}
There are two interesting antecedents for the results we have just described: one is a theorem of Borthwick-Paul-Uribe having to do with Toeplitz operators on line bundles over projective varieties (see \cite{BPU}) and the second an analogue of the theorem above for semi-classical pseudodifferential operators  (see \cite{GS} chapter 12, theorem 12.13.1.)
\end{remark}

We will conclude this introduction by giving a quick overview of the contents of sections 2-8: In section 2 we will recall how classical Szeg\"o projectors and Toeplitz operators are defined and in section 3 we will describe the symbol calculus of these operators. Then in section 4 we'll transition to the semi-classical setting and give a brief account of the ``alternative approach" to defining Toeplitz operators that was described in \S 3.4 of \cite{GUW} and use this approach to show that the $P_\hbar$'s above are operators of this type. Then in section 5 we will prove Theorem \ref{thm1} and in section 6 deduce Theorem \ref{thm2} from Theorem \ref{thm1}. 

Finally in sections 7 and 8 we will discuss some applications of Theorem \ref{thm2} to inverse  problems involving the ``equivariant spectrum" of the Toeplitz operator (\ref{ToepOpe}), i.e. the spectral invariants of this operator associated with the asymptotic expansion (\ref{PiPkalphaPiExp}). More explicitly in section 7 we will show that for Toeplitz operators on the boundary of the unit ball in $\mathbb C^n$, the symbols of these operators are e-spectrally determined and in section 8 we will show by a ``reduction in stages" argument that this result implies an analogous result for Toeplitz operators on arbitrary toric varieties.

A concluding remark: Though the results of this paper are officially  theorems in semi-classical analysis, the theorems in sections 7 and 8  can also be interpreted as illustrations of the ``quantization commutes  with reduction" principle for $G$-actions on symplectic manifolds, a  principle which had its origins in the work of Bert Kostant and his  collaborators, and for this reason we deem it an honor to be able to  contribute this paper to the Kostant memorial volume.

\section{The classical theory of Toeplitz operators}

The Toeplitz operators that one studies in complex analysis are the operators that we alluded to in the introduction (i.e. the Toeplitz operators on the boundary of a strictly pseudoconvex domain) and the ``standard model" that one usually chooses for this class of operators is the algebra of Toeplitz operators on the unit ball $B^n$ in $\mathbb C^n$. However, since we are concerned in this article not only with the Toeplitz operators that come up in complex analysis, but also with Toeplitz operators that come up in other contexts as well we will choose a somewhat simpler ``standard model". Namely we will choose for the manifold $X$ in \S 1 the product space $\mathbb R^k \times \mathbb T^l$, where $\mathbb T^l$ is the standard $l$ dimensional torus
\[
S^1 \times \cdots \times S^1 \qquad (l \text{\ factors})
\]
and $k+l=n$, and define our Hardy space $H^2(X)$ to be the space of $L^2$ solutions of the system of equation 
\begin{equation}\label{modelope}
\frac 1{\sqrt{-1}} \left( \frac{\partial}{\partial y_j}+y_j |D_\theta|\right)f=0,\qquad j=1, \cdots, k,
\end{equation}
where $|D_\theta|$ is the square root of the Laplace operator on $\mathbb T^l$. For each $0 \ne m \in \mathbb Z^l$ the function 
\begin{equation}
\label{fmfunction}
 f_m=e^{-\frac{|y|^2|m|}2}\left(\frac{|m|}{\pi}\right)^k e^{im\theta}
 \end{equation}
is a solution of (\ref{modelope}) and these functions form an orthogonal basis of $H^2(X)$. In particular the map 
\begin{equation}\label{operatorR}
 R: L^2(\mathbb T^l) \to L^2(X)
\end{equation}
mapping $e^{im\theta}$ onto $f_m$ is an orthonormal embedding of $L^2(\mathbb T^l)$ onto $L^2(X)$ and if we denote by $\Pi$ the orthogonal projection 
\begin{equation}\label{OrthProPi}
 L^2(X) \to H^2(X)
\end{equation}
the map $R$ satisfies the identities 
\begin{equation}\label{RtR}
R^tR=\mathrm{Id}
\end{equation}
and
\begin{equation}
\label{RRt}
RR^t=\Pi.
\end{equation}

Finally we note that if $\eta$ and $\tau$ are the dual variables to $y$ and $\theta$ in the cotangent bundle of $\mathbb R^k \times \mathbb T^l$, then the symbol of the operator (\ref{modelope})  is 
\begin{equation}\label{symbmodel}
\frac 1{\sqrt{-1}} \eta_j + y_j |\tau|
\end{equation}
and hence the characteristic  variety of the system of equation defining $H^2(X)$ is 
\begin{equation}\label{yetazero}
y=\eta=0,
\end{equation}
i.e. is the image of the embedding 
\begin{equation}\label{embcotan}
T^*\mathbb T^l \to T^*X
\end{equation}
that maps $T^*\mathbb T^l$ onto the second factor of the product 
\begin{equation}\label{factorTstar}
T^*X = T^*\mathbb R^k \times T^* \mathbb T^l.
\end{equation}
In particular this characteristic variety is the symplectic cone 
\begin{equation}\label{symplconeSig}
\Sigma = T^*\mathbb T^l
\end{equation}
in the product $T^*(\mathbb R^k \times \mathbb T^l)$ and hence the characteristic variety of the projection $\Pi$ is the diagonal 
\begin{equation}\label{DeltaSigma}
\Delta_\Sigma=T^*\mathbb T^l
\end{equation}
in the product $T^*\mathbb T^l \times T^*\mathbb T^l$.

Turning to the topic of this section: ``Toeplitz operators". These are by definition operators on $H^2(X)$ of the form $\Pi P \Pi$ where $P$ is a classical pseudodifferential operator on $X$. Alternatively they can be described as operators of the form 
\[RQR^t,\]
where $Q: C^\infty(\mathbb T^l) \to C^\infty(\mathbb T^l)$ is a classical pseudodifferential operator on $\mathbb T^l$ or as an operator of the form $P\Pi$, where $P$ is a classical pseudodifferential operator on $X$ that commutes with $\Pi$ (For the equivalence of these three definitions see \cite{BG}. We note that by definition 3 these operators form an algebra, and that by definition 2 this algebra is isomorphic to the algebra of classical pseudodifferential operators on $\mathbb T^l$.) 

With these operators as our ``standard model" we will define Toeplitz operates in general to be operators which are microlocally isomorphic to the operators in this standard model. More explicitly let $X$ be an $n$ dimensional manifold and $\Sigma \subset T^*X-0$ a symplectic cone. We will define a ``Toeplitzification" of $\Sigma$ to be a self-adjoint linear operator $\Pi: L^2(X) \to L^2(X)$ with the properties 
\begin{enumerate}
\item $\Pi^2=\Pi$,
\item The microsupport of $\Pi$ is the diagonal $\Delta_\Sigma$ in $\Sigma \times \Sigma$,
\item At every point $(p, \xi) \in \Delta_\Sigma$, $\Pi$ is microlocally isomorphic to the ``$\Pi$" in the canonical model above.
\end{enumerate}
Moreover, if one has an $S^1$ action on $X$ and the induced action on $L^2(X)$ commutes with $\Pi$, we will call this an \emph{equivariant Toeplitz structure} if the isomorphism above conjugates this $S^1$ action to the action of a circle subgroup of $\mathbb T^l$ on the canonical model. 

\section{The symbol calculus for Toeplitz operators}

The operator $\Pi P \Pi$ that we defined in \S 2 is an example of a ``Fourier integral operator of Hermite type" and the symbol calculus for these operators is fairly complicated. (In particular, it involves ``symplectic spinors" objects which we won't attempt to define except to say that they are section of a vector bundle over $\Sigma$ whose fiber at $p \in \Sigma$ is an infinite dimensional vector space.) Fortunately however, for Teoplitz operators there is an alternative symbol calculus that is much simpler and is based on the following result:
\begin{theorem}\label{thm31}
Let $P: C^\infty(X) \to C^\infty(X)$ be a classical pseudodifferential operator of order $m$ and let $Q=\Pi P \Pi$. Then if $p: T^*X-0 \to \mathbb C$ is the leading symbol of $P$, its restriction 
\begin{equation}\label{symbolrestr}
p|_\Sigma = q
\end{equation}
is an intrinsically defined symbol of $Q$, i.e. doesn't depend on the extension $P$ of $Q$ to $L^2(X)$. Moreover, if this symbol vanishes there exists a pseudodifferential operator $P': C^\infty(X) \to C^\infty(X)$ of order $m-1$ with $Q=\Pi P' \Pi$.
\end{theorem}
\begin{proof}
To prove this it suffices to prove this for the ``standard model" described in \S 2. However, for the standard model the symbol we've just described, the symbol of $Q$ is the symbol of the classical pseudodifferential operator $RQR^t$ and hence only depends on the restriction of $p$ to $\Sigma$.
\end{proof}

\section{The semi-classical theory of Toeplitz operators}

The approach to this subject that we will describe below is motivated by the alternative approach to the theory of semi-classical pseudodifferential operators developed by T. Paul and A. Uribe in \cite{PU}. To describe this approach let $Y$ be a manifold, let $X$  be the product manifold $Y \times S^1$, and let 
\begin{equation}\label{Pope}
P: C^\infty(X) \to C^\infty(X)
\end{equation}
be a classical pseudodifferential operator which is invariant with respect to the action of $S^1$ on $Y \times S^1$. Its symbol then has the form $p(y, \eta, \tau)$, where $\tau$ is the dual cotangent variable to the angle variable $\theta$. Now let
\begin{equation}
\label{L2decom}
L^2(X) = \bigoplus_k L^2(X)_k
\end{equation}
be the orthogonal decomposition of $L^2(X)$ into the sum of the subspaces, $L^2(X)_k$, of functions in $L^2(X)$ which transform under the action of $S^1$ by the recipe 
\begin{equation}\label{weightkfunc}
(\tau_\theta f)(x) = e^{ik\theta}f(x).
\end{equation}
Then the operator 
\begin{equation}\label{Phbark}
P_\hbar=\hbar^lP|_{L^2(X)_k}
\end{equation}
can be thought of as a semi-classical pseudodifferential operator of order $m-l$. Moreover, via the identification 
\begin{equation}\label{IdYXk}
L^2(Y) \to L^2(X)_k
\end{equation}
mapping $f(y)$ to $f(y)e^{ik\theta}$ it can be regarded as an operator on $Y$. (For instance if $X$ is equipped with a Riemannian metric $\sum g_{ij}(y) dy_idy_j+V(y)(d\theta)^2$, then for $l=2$ the Laplace operator on $X$ gets converted by this process into the Schr\"odinger operator
\begin{equation}\label{SchroOpe}
\hbar^2 \Delta_Y +V(y)
\end{equation}
on $Y$.) Another key definition in this approach to theory of semi-classical pseudodifferential operators is the notion of symbol for the operators above. The semi-classical symbol of (\ref{Phbark}) being the function 
\begin{equation}\label{p}
p(y, \eta, 1), 
\end{equation}
where $p(y, \eta, \tau)$ is the symbol of the classical $S^1$-invariant pseudodifferential operator $P$, $\eta$ being the dual cotangent variable to $y$ and $\tau$ the dual cotangent variable to $\theta$. 

Turning finally to the topic of this paper, this alternative approach to the semi-classical theory of pseudodifferential operators can also be applied to Toeplitz operator by mimicking the definition above of semi-classical pseudodifferential operators. Namely if $Q$ is the classical Toeplitz operator $\Pi P \Pi$, its semi-classical counterpart is the analogue 
\begin{equation}\label{Qhbark}
Q_\hbar=\hbar^l Q|_{H^2(X)_k},
\end{equation}
where $H^2(X)_k$ is the intersection of $H^2(X)$ with $L^2(X)_k$. 

Finally we note that the theory above can be generalized slightly by allowing the manifold $X=Y \times S^1$ to be replaced by a manifold $X$ on which one has a free action of the circle group $S^1$. In this case the semi-classical operators $P_\hbar$ and $Q_\hbar$ can no longer be thought of as operators on $Y$, but one can get around this problem either by thinking of them as operators on sections of $\mathbb L^k$, where $\mathbb L^k$ is the line bundle 
\begin{equation}\label{Lkquotient}
X \times \mathbb C /S^1,
\end{equation}
$S^1$ acting on $\mathbb C$ by multiplication by $e^{ik\theta}$, or by thinking of them in terms of local trivialization of $\mathbb L$ over open subsets of $X$. (In the next two sections of this paper we will for the most part adopt this second approach.)

\section{The proof of Theorem \ref{thm1}}

Let $\{\mathcal U_i, i=1, \cdots, N\}$ be a covering of $\Sigma$ by conic open subsets of $T^*X-0$ having the property that on $\mathcal U_i$ the Toeplitz projector, $\Pi$, is conjugated by an invertible $S^1$-equivariant Fourier integral operator to the Toeplitz projector in the canonical model that we described in section 2. In addition let 
\[
\{P_i, i=1, \cdots, N\}
\]
be an $S^1$-equivariant microlocal partition of unity subordinate to this cover (i.e. each $P_i$ is a classical zeroth order pseudodifferential operator which is $S^1$ invariant and has microsupport on $\mathcal U_i$; and the sum of these $P_i$'s is microlocally equal to the identity operator on a conic open neighborhood of $\Sigma$.) Thus to compute the trace of $\Pi f(P_\hbar) \Pi$ it suffices to compute the traces of each of the summands in the sum 
\begin{equation}\label{sumlocal}
\sum_i \Pi P_i f(P_\hbar) \Pi
\end{equation}
and to compute the trace of $\Pi P_i f(P_\hbar) \Pi$ it suffices, by conjugation by an invertible $S^1$ equivariant Fourier integral operator to compute the trace of the operators corresponding to $\Pi P_i f(P_\hbar) \Pi$ in the canonical model. However, in this canonical model we can assume that the $P_i$'s and $f(P_\hbar)$ commute with $\Pi$ and hence that this trace is identified with the trace of the semi-classical pseudodifferential operator 
\begin{equation}\label{tildePif}
\widetilde P_i f(\widetilde P_\hbar),
\end{equation}
where $\widetilde P_i=R^t P_i R$, and $\widetilde P_\hbar=R^t P_\hbar R$, $R$ being the ``Fourier integral operator of Hermite type" defined by (\ref{operatorR}). However, the $\widetilde P_i$'s are classical pseudodifferential operates and $\widetilde P_\hbar$  a semi-classical pseudodifferential operator. Therefore since $\widetilde P_i$ is $S^1$ equivariant,  $\widetilde P_i f(\widetilde P_\hbar)$ is a semi-classical pseudodifferential operator as well; so to compute the trace of the $i^{th}$ summand of (\ref{sumlocal}) we are reduced to computing the trace of this pseudodifferential operator. Fortunately, however, the pseudodifferential version of the theorem we are trying to prove is well-known; i.e. one has an asymptotic expansion of this trace in the power of $\hbar$ which is identical with the expansions (\ref{muhexpansion}) and (\ref{c0f}) except for the fact that $\Sigma$ is now the cotangent bundle of $\mathbb T^l$ rather than the factor $T^*\mathbb T^l$ in the product (\ref{factorTstar}) (and except for the fact that one has to insert a factor $\rho_i$ in the integral (\ref{c0f}),  $\rho_i$ being the symbol of $P_i$.) Thus, conjugating the operators $\Pi P_i f(P_\hbar) \Pi$ in the $i^{th}$ summand of (\ref{sumlocal}) into the canonical model and summing over $i$ we get the asymptotic expansions (\ref{muhexpansion}) and (\ref{c0f}).

\begin{remark}
For a proof of the pseudodifferential version of Theorem 1 see \cite{GS}, \S 12.12, page 330-331. Also, for a somewhat simpler proof of an enlightening special case of this theorem in which $P_\hbar$ is Schr\"odinger operator $\hbar^2 \Delta+V$, see \cite{DGS} (and for some inverse spectral applications of this result see \cite{GW}).
\end{remark}

\section{The proof of Theorem \ref{thm2}}

We will briefly describe how the expansion (\ref{PiPkalphaPiExp}) can be deduced from the expansion (\ref{muhexpansion}). Let $\mathfrak g$ be the Lie algebra of $G$ and let 
\[\mathfrak g_\alpha=\{ v \in \mathfrak g :  \alpha(v)=0\}.\]
This ``$\mathfrak g_\alpha$" is the Lie algebra of a codimension one subtorus, $G_\alpha$, of $G$ and assuming that this subtorus acts freely on $X$, the quotient manifold 
\[Y=X/G_\alpha\]
is well defined and the projection 
\begin{equation}\label{proXY}
\pi: X \to Y
\end{equation}
a smooth fibration. Let $\pi^*T^*Y$ be the ``horizontal subbundle" of $T^*X$ with respect to this projection. Then one gets from (\ref{proXY}) a fibration
\begin{equation}\label{pullproXY}
\pi^*T^*Y \to T^*Y
\end{equation}
mapping the intersection $\Sigma \cap \pi^* T^*Y$ onto a symplectic cone, $\Sigma_{\alpha, 0}$, in $T^*Y$ and this cone can easily be seen to be the symplectic reduction of $\Sigma$ with respect to the action of $G_\alpha$ on the zero level set of its $G_\alpha$ moment map. 

Moreover, this reduced space has a residual action of the circle group $G/G_\alpha$, and it is easily seen that the symplectic reduction of $\Sigma_{\alpha, 0}$ with respect to the action of this circle group at the ``one" level set of its moment map is $\Sigma_\alpha$. 

We'll next describe the quantum analogue of this picture. Let $H^2(Y)$ be the space of $G_\alpha$-invariant functions in $H^2(X)$. Then the orthogonal projection 
\begin{equation}\label{piyyx}
\pi_Y: L^2(Y) \to H^2(Y)
\end{equation}
is a Toeplitz projector with microsupport on the diagonal in $\Sigma_{\alpha,0} \times \Sigma_{\alpha,0}$. Moreover, there is a residual action of the circle group $G/G_\alpha$ on this space and the corresponding decomposition into weight spaces 
\begin{equation}
H^2(Y) = \bigoplus_k H^2(Y)_k
\end{equation}
coincides with the sum 
\begin{equation}
\bigoplus_k H^2(X)_{k\alpha},
\end{equation}
and as a consequence of these observations the asymptotic formula (\ref{PiPkalphaPiExp}) can easily be deduced from the asymptotic formula (\ref{muhexpansion}) with $X$ replaced by $Y$.

\section{The basic example}

As we mentioned in the introduction to \S 2 the ``basic example" of a Szeg\"o projector for complex analysts is the example associated with the boundary, $S^{2n-1}$, of the unit ball, $B^n$, in $\mathbb C^n$; i.e. in this example $H^2(S^{2n-1})$ is the subspace of $L^2(S^{2n-1})$ spanned by the functions 
\begin{equation}\label{HSS2n-1}
f|_{S^{2n-1}}, \quad f \in \mathcal O(B^n) \cap C^\infty(\overline B^n), 
\end{equation}
and the symplectic cone $\Sigma$ in $T^*S^{2n-1}$ can be identified with $\mathbb C^n-0$, the symplectic form on this cone being the standard K\"ahler form 
\begin{equation}
\omega = \sqrt{-1} \sum_i dz_i \wedge d\bar z_i. 
\end{equation}
Moreover, $\omega$ is preserved by the action on $\mathbb C^n-0$ of the $n$-torus 
\begin{equation}
T^n = S^1 \times \cdots \times S^1
\end{equation}
and this action is a Hamiltonian action with moment map 
\begin{equation}\label{mmapT}
\phi(z_1, \cdots, z_n) = (|z_1|^2, \cdots, |z_n|^2)
\end{equation}
and hence one gets other interesting examples by reduction by subgroups $G$ of $T^n$. We'll focus in this section on the simplest of these examples, reduction by $T^n$ itself, and in this case we'll show that the asymptotic formula (\ref{PiPkalphaPiExp}) in Theorem \ref{thm2} is essentially a formula for the eigenvalues of the operator $Q_\hbar=  \Pi P_\hbar \Pi$. 

The moment map associated  with the $T^n$ action on $\mathbb C^n$ is the mapping (\ref{mmapT}). Therefore for $\alpha = (\alpha_1, \cdots, \alpha_n) \in \mathbb Z_+^n$, $\phi^{-1}(z)=\alpha$ if and only if $|z_i|^2 = \alpha_i$. Hence if the $\alpha_i$'s are all non-zero, $T^n$ acts freely on the space $\phi^{-1}(\alpha)$ and the reduced space 
\begin{equation}\label{redspatalpha}
\phi^{-1}(\alpha)/T^n
\end{equation}
consists of a single point, $P_\alpha$. Moreover, the quantum analogue, $H^2(X)_\alpha$, of this reduced space is the set of holomorphic functions $f \in \mathcal O(B^n)$ which transform under the action of $T^n$ by the weight $\alpha$, i.e. satisfy 
\begin{equation}
f(e^{i\theta_1}z_1, \cdots, e^{i\theta_n}z_n) = e^{i\alpha(\theta)}f(z)
\end{equation}
and this is just the one dimensional subspace of $\mathcal O(B^n)$ spanned by $z_1^{\alpha_1} \cdots z_n^{\alpha_n}$. However, these are exactly the eigenspaces of $Q_\hbar$;  so the asymptotic formula (\ref{PiPkalphaPiExp}) in Theorem \ref{thm2} is just an asymptotic formula for the equivariant eigenvalues of $Q_\hbar$. 


We will next show that, as one application of this formula, we get the following inverse spectral result.
\begin{theorem}
The symbol of the Toeplitz operator $Q$ is e-spectrally determined. 
\end{theorem}
\begin{proof}
The leading term in the asymptotic expansion (\ref{PiPkalphaPiExp}) determines the symbol of $Q_\hbar$ on the set (\ref{redspatalpha}); however, this set consists of the single point $P_\alpha$; so it in effect determines the symbol of $Q_\hbar$ at $P_\alpha$ and the symbol $p$ of $P$ on the set, $\Phi^{-1}(\alpha)$. Moreover, since $P$ is a classical pseudo-differential operator on $S^{2n-1}$ of order zero, $p$ is a homogeneous functions  of order zero on the punctured cotangent bundle $T^*S^{2n-1}-0$, and its restriction to the symplectic cone $\Sigma = \mathbb C^n-0$ is a homogeneous function of degree zero on $\Sigma$. Thus if $z$ is an element of $\phi^{-1}(\alpha)$, i.e. $(|z_1|^2, \cdots, |z_n|^2)=\alpha$, then $p$ takes the same value at $z$ and at $tz$, so $p$ is spectrally determined on the set of $z$'s satisfying 
\begin{equation}
(|z_1|^2, \cdots, |z_n|^2) = t\alpha
\end{equation}
and since this contains the set of points $(z_1, \cdots, z_n) \in \mathbb C^n-0$ for which the $|z_i|$'s are rational numbers it is a dense subset of $\mathbb C^n-0$, so it follow that the symbol of $\Pi P \Pi$ is e-spectrally determined. 
\end{proof}

\begin{remark}
In some interesting cases this become an e-spectral result for $Q$ itself. For instance this is the case for Toeplitz operators of the form $\Pi M_F \Pi$, where $M_F: L^2(S^{2n-1}) \to L^2(S^{2n-1})$ is the operator ``multiplication by a function $F \in C^\infty(S^{2n-1})$".
\end{remark}

\section{Toric varieties}

By a reduction in stages argument we will show in this section how to extract from the example above inverse spectral  results for a much larger class of example: toric varieties. To describe these examples let $G$ be a subtorus of $T^n$ and let 
\begin{equation}
\phi: \mathbb C^n-0 \to \mathfrak g^*
\end{equation}
be the moment map associated with the action of $G$ on $\mathbb C^n-0$. Then if $\alpha$ is a regular value of this map and $G$ acts freely on $\phi^{-1}(\alpha)$, the symplectic reduction 
\begin{equation}\label{quoG}
M =\phi^{-1}(\alpha) / G
\end{equation}
is well-defined and inherits from $\mathbb C^n-0$  a K\"ahler structure. Moreover, from the action of $T^n$ on $\mathbb C^n-0$ it inherits a Hamiltonian action of the quotient torus $K=T^n/G$. 
\begin{lemma}
The $K$-action on $M$ is a toric action.
\end{lemma} 
\begin{proof} 
Since the Lie algebra $\mathfrak k$ of $K$ is a quotient algebra of the Lie algebra $\mathfrak t_n$ of $\mathbb T^n$ one has an inclusion of $\mathfrak k^*$ into $\mathfrak t_n^*$, and if $\beta \in \mathfrak k^*$ is a regular value of the $K$ moment map of $M$ into $\mathfrak k^*$, its image $\delta$ in $\mathfrak t_n^*$ is a regular value of the $T^n$ moment map of $\mathbb C^n-0$ into $\mathfrak t_n^*$. Moreover, the symplectic reduction of $M$ with respect to $\beta$ coincides with the symplectic reduction of $\mathbb C^n-0$ with respect to $\delta$. However, since the action of $\mathbb T^n$ on $\mathbb C^n-0$ is toric the second of these reduced spaces is just a point and hence so is the first. 
\end{proof}

Thus by a theorem of Delzant, the moment map 
\begin{equation}\label{mmapK}
\psi: M \to \mathfrak k^*
\end{equation}
defined by this action has the property that if $\beta$ is a regular value of $\psi$, the reduced space $\phi^{-1}(\beta)/K$ is just a point $p_0$. In fact the image of the map (\ref{mmapK}) is the moment polytope $\Delta_M$ of $M$ and $\beta \in \Delta_M$ is a regular value of $\psi$ if and only if $\beta$ is in the interior $\mathring{\Delta}_M$ of $\Delta_M$. Thus if $M_0$ is the set of points in $M$ at which $K$ acts freely, one gets from (\ref{mmapK}) a bijective map 
\begin{equation}
M_0/K \to \Delta_0.
\end{equation}
Hence one can identify the points $p_0$ above with elements of $\mathring{\Delta}_M$. 

Coming back to the definition (\ref{quoG}). Suppose that the ``$\alpha$" in this definition is an element of the weight lattice of $G$, and let $\mathfrak g_\alpha$ be the annihilator of $\alpha$ in $\mathfrak g$. Then $\mathfrak g_\alpha$ is the Lie algebra of a subtorus $G_\alpha$ of $G$. Now let
\begin{equation}
\tilde \varphi_\alpha: T^*S^{2n-1} \to \mathfrak g_\alpha^*
\end{equation}
be the moment map associated to the lifted Hamiltonian action of $G_\alpha$ on $T^*S^{2n-1}$. Then its restriction 
\begin{equation}
\varphi_\alpha: \mathbb C^n-0 \to \mathfrak g_\alpha^*
\end{equation}
is, by definition, the moment map associated with the action of $G_\alpha$ on this cone. Thus the symplectic reduction $\tilde \varphi_\alpha^{-1}(0)/G_\alpha$ is the cotangent space of the quotient manifold 
\[ Y = S^{2n-1}/G_\alpha \]
and the reduced symplectic cone 
\begin{equation}
C=\varphi_\alpha^{-1}(0)/G_\alpha
\end{equation}
sits inside $T^*Y$ as a symplectic sub-cone. Moreover one has a residual action of the group $S^1=G/G_\alpha$  on this cone, and if 
\begin{equation}
\gamma: C \to \mathbb R
\end{equation}
is the moment map associated with this action, the symplectic reduction 
\begin{equation}
\gamma^{-1}(1)/S^1
\end{equation}
is another description of $M$ and the symplectic reduction 
\begin{equation}
\gamma^{-1}(k)/S^1
\end{equation}
is the manifold $M$; but now with its symplectic form $\omega_M$ replaced by $k \omega_M$. 

Now let $H^2(S^{2n-1})$ be the Hardy space (\ref{HSS2n-1}) and let $H^2(Y)$ be the space of $G_\alpha$-invariant elements in $H^2(S^{2n-1})$ and 
\begin{equation}\label{H2Yl}
H^2(Y) = \bigoplus_l H^2(Y)_l
\end{equation}
the decomposition of this space into weight spaces with respect to the residual action of $S^1$ on $H^2(Y)$. These can be regard as quantizations of the symplectic manifold $(M, l\omega)$ and the quantum analogue of the assertion that the action of $K$ on $M$ is \emph{toric} is 
\begin{lemma}
The action of $K$ on the summands of (\ref{H2Yl}) are \emph{multiplicity free}. In other words, the weight spaces of the representation of the torus $K$ on the space $H^2(Y)_l$ are one dimensional. 
\end{lemma}
\begin{proof} 
By definition $K$ is the quotient group $T^n/G$, so one has a natural inclusion map $\iota: \mathbb Z_K \hookrightarrow \mathbb Z^n$ of the weight lattice of $K$ into the weight lattice of $\mathbb T^n$. Moreover, for $\beta \in \mathbb Z_K$ and $\delta=\iota(\beta)$ the weight space $H^2(S^{2n-1})_\delta$ coincides with the weight space $H^2(Y)_\beta$, and we showed in section 1 that the first of these spaces is  one dimensional, hence so is the second. 
\end{proof}

As a consequence we prove  
\begin{theorem}
Let $\Pi$ be the orthogonal projection of $L^2(Y)$ onto $H^2(Y)$, $P: C^\infty(Y) \to C^\infty(Y)$ a $K$-invariant zeroth order pseudodifferential operator and $Q: H^2(Y) \to H^2(Y)$ the associate Toeplitz operator 
\begin{equation}
Q=\Pi P \Pi. 
\end{equation}
Then the symbol of $Q$ is e-spectrally determined. 
\end{theorem}
\begin{proof}
Since the action of $K$ on the symplectic manifold $(M, l\omega)$ is a toric action, the symplectic manifold (\ref{sigmaalpha}) in Theorem \ref{thm2} are just points. Moreover, these points can be identified with the lattice points of the moment polytope $l\Delta_M$. In addition, since the action of $K$ on $H^2(Y)_l$ is multiplicity free, the asymptotic expansions (\ref{PiPkalphaPiExp}) are just asymptotic expansions for the eigenvalues of the semiclassical Toeplitz operator $Q_\hbar$, $\hbar=1/l$. Moreover, by ``multiplicity free" the submanifolds (\ref{sigmaalpha}) are just points and thus leading terms (\ref{mualpha0}) in these expansions are the restrictions of $p$ to the pre-image with respect to the moment map of the lattice points in $l\mathring{\Delta}_M$. However, the union of these $l\mathring{\Delta}_M$'s is the entire weight lattice $\mathbb Z_K$ of $K$, and the symbol $p$ is a homogeneous function of order zero, so it is e-spectrally determined on the preimage of $\mathbb R^+\mathbb Z_K$ (and, since $\mathbb R^+ \mathbb Z_K$ is dense in $\mathfrak k^*$, is therefore spectrally determined.) 
\end{proof}

\end{document}